\documentclass[12pt]{article}
\usepackage[english]{babel}
\usepackage{amsthm,amsmath,amssymb,graphicx,latexsym,amsthm}
\usepackage{hyperref}
\hypersetup{
    colorlinks=true,
    linkcolor=blue,
    pdftitle={Hamilton-Jacobi-Bellman equations for Chemical Reaction Networks},
    bookmarks=true,
    pdfpagemode=FullScreen,
}

\newtheorem{theorem}{Theorem}

\newtheorem{corollary}[theorem]{Corollary}

\newtheorem{convention}[theorem]{Convention}

\newtheorem{proposition}[theorem]{Proposition}
\newtheorem{remark}[theorem]{Remark}

\numberwithin{equation}{section}
\numberwithin{theorem}{section}

\theoremstyle{definition}
\newtheorem{example}{Example}
\numberwithin{example}{section}
\newtheorem{definition}[theorem]{Definition}

\title{Hamilton-Jacobi-Bellman equations for Chemical Reaction Networks}
\author{Michael Snarski}

\begin{document}
\maketitle

\begin{abstract}
This is an expository note on large deviations, Hamilton-Jacobi-Bellman (HJB) equations, and the role of the Freidlin-Wentzell quasipotential in chemical reaction networks (CRNs). The note was motivated by observations in \cite{AndersonLyapunov15} which identified Lyapunov functions for deterministic descriptions of CRNs by taking appropriate scaling limits of invariant distributions for the corresponding stochastic dynamics. We explain how this is a special case of a classical theory due to Freidlin and Wentzell \cite{FreidlinWentzell}. We also show that this Lyapunov function is a solution to the HJB partial differential equation if and only if the network is ``complex-balanced''. The target audience are researchers in the CRN community who are familiar with the Markov process description of CRNs, but do not have the time to invest in learning the technical machinery of large deviations.  We conclude by exploring some possible relationships with K{\"a}hler geometry which suggest an interesting and unexplored research direction. 
\end{abstract}

\section{Introduction}

Chemical reaction networks (CRNs) attempt to model the behavior of chemical
species interacting in a well-stirred homogeneous medium of large but finite volume. 
Inflow and outflow of species is allowed, and we assume that the time-scale
over which particles meet is ``sufficiently small'' so that they are
guaranteed to interact ``often''. Under such assumptions every particle of the
same species experiences identical conditions. As such, CRNs do not
track the spatial location of any individual particle but rather provide only
macroscopic information. In this note, the macroscopic description will be either a
discrete number of a certain type of species (if a \emph{stochastic} description of
the system is employed) or a concentration value, which is a non-negative real
number (for a \emph{deterministic} description).

Our discrete, stochastic description of CRNs is a sequence of jump
Markov processes taking values in increasingly finer lattice subsets of
$\mathbb{R}^n_+$. The sequence is indexed by a scaling parameter that
represents the volume of the container. We will denote the scaling parameter
by $n \geqslant 1$, and the sequence of processes by $\{ X^n (\cdot) \}_{n
\geqslant 1}$. Under appropriate conditions on the jump intensities, this
sequence has a Law of Large Numbers {\cite{KurtzLLN77}}{\cite{KurtzCRN72}}
limit which we denote by $x (\cdot)$, a result we sometimes refer to as
\emph{Kurtz's theorem}. The limit $x (\cdot)$ is absolutely continuous and
can be characterized as the unique solution to a system of coupled ordinary
differential equations; this defines the continuous,
deterministic description {\cite{KurtzCRN72}}.

The focus of this paper is to explain some aspects of the relationship between the
stochastic and deterministic descriptions. An important class of networks for
us will be the \emph{complex-balanced} ones, introduced by Horn
and Jackson {\cite{HornJackson72}}. They showed that if $c \in
\mathbb{R}^d_{> 0}$ is a steady state for the deterministic dynamics of a
complex-balanced network, then the function of $x$ defined by
\begin{equation}
  V (c, x) \doteq \sum_{i = 1}^d c_i \ell \left( \frac{x_i}{c_i} \right),
  \quad \ell (x) \doteq x \log x - x + 1 \label{eqn:lyapunov_function}
\end{equation}
is a Lyapunov function for the deterministic dynamics.

It is already known that there is a relationship between the Lyapunov
function $V$ and the stochastic dynamics. In {\cite{AndersonLyapunov15}},
the authors explicitly write down the stationary distribution associated to
the Markov process $X^n$ using results from {\cite{AndersonProductForm10}},
and then take an appropriate scaling limit of the stationary distributions to
obtain $V$. \

It turns out that this procedure is a special case of a well-established
result in a subfield of probability known as Freidlin-Wentzell theory. The
trouble with this theory is that it can seem prohibitively technical to the
non-specialist. The goal of the present paper is to explain in the simplest
terms some implications of the theory, communicate the intuition behind it,
and provide ways of leveraging the results without knowing all the details.

Here is one contribution of the approach. We will associate to any CRN a
function $L : \mathbb{R}^d_+ \times \mathbb{R}^d$, called the \emph{local
rate} or \emph{Lagrangian}, and study the family of variational problems
\begin{equation}\label{eqn:quasipotential}
 Q (c, x) = \inf_{\varphi} \left\{ \int_0^T L (\varphi (s), \dot{\varphi}
   (s)) \text{ds} : \varphi (0) = c, \varphi (T) = x, T < \infty \right\}, 
\end{equation}
where the infimum is over all absolutely continuous paths $\varphi : [0, T]
\rightarrow \mathbb{R}^d_+$. We will show that $Q (c, x)$ is always a Lyapunov
function for the deterministic dynamics, and $Q (c, x) = V (c, x)$ if and only
if the network is complex-balanced.

A key part of our exposition is to emphasize a ``Hamiltonian description'' of a
chemical reaction network. In the case of mass-action kinetics, there is a
one-to-one correspondence between reaction networks and their naturally
associated Hamiltonians. In turn, each Hamiltonian defines an associated partial differential equation, known as a Hamilton-Jacobi-Bellman (HJB) equation. We will demonstrate how the steady-state HJB can be used to study the stability of both the stochastic and deterministic descriptions. 

The layout of the paper is as follows. Section \ref{sec:crn_hamiltonian}
serves to define the most basic objects in the paper, the CRNs and their
associated Hamiltonian. In Section \ref{sec:ldp_to_classical_mechanics} we
sketch the basic framework connecting CRNs to large deviations and classical
mechanics. We do so in three steps: in Section \ref{sec:crn_to_lln} we show
how the Hamiltonian encodes the Law of Large Numbers (LLN) limit; in Section
\ref{sec:lln_to_ldp} we study the rate of convergence of the LLN and introduce
the notion of a Large Deviations Principle (LDP); and in Section
\ref{sec:ldp_to_hjb} we show how the LDP is related to the
HJB equation through a family of variational
problems.

Once we have established the main objects of large deviations theory, we
focus in Section \ref{sec:quasipotential} on a specific object known as the
\emph{quasipotential}, defined above in \eqref{eqn:quasipotential}. In Section \ref{sec:quasipotential_lyapunov} we
show that the quasipotential serves as a Lyapunov function, and in Section
\ref{sec:quasipotential_invariant} we show that the invariant measure
technique employed in {\cite{AndersonLyapunov15}} identifies the
quasipotential. It is therefore automatic that all the functions identified in
{\cite{AndersonLyapunov15}} are Lyapunov functions. In Section
\ref{sec:quasipotential_hjb}, we show that the quasipotential can be
identified using a partial differential equation known as the
Hamilton-Jacobi-Bellman (HJB) equation.

The results described so far hold for general reaction networks with mass-action kinetics. In Section
\ref{sec:complex_balanced} we specialize the results to complex-balanced
networks and show that the function $V$ solves the HJB equation if and only if
the network is complex-balanced. Finally, Section \ref{sec:examples} shows
some worked examples. Section \ref{sec:examples_anderson} aims to demonstrate
how the HJB technique can be simpler than the one employed in
{\cite{AndersonLyapunov15}}. On the other hand, Section
\ref{sec:examples_kahler} shows a curious coincidence between the form of $V$
and the form of a potential for a K{\"a}hler metric obtained from the
combinatorial data of a polytope in {\cite{Abreu08}}. We suspect that there is
a deeper connection between the two, and that this is an interesting area of
research at the intersection of toric geometry and convex analysis. In
particular, the fan description of a toric variety may be closely related to
the fans used in work of Craciun {\cite{Craciun15}}.

We re-iterate that this paper serves as
an explanatory guide, and contains only a few novel observations. In particular, we do not go into detail about the theory of large deviations or viscosity solutions to HJB equations. A good rigorous reference on the geometry of CRNs within
the context of Freidlin-Wentzell theory is {\cite{Agazzi18u}}. The definitive
book on Freidlin-Wentzell theory is {\cite{FreidlinWentzell}}. The topic of solutions to HJB equations is a vast and subtle one and relies heavily on the seminal work of Crandall and Lions on viscosity solutions \cite{Crandall83} (see also \cite{Bardi97} and \cite{Feng06}). 

\

\section{CRNs and their Hamiltonians}\label{sec:crn_hamiltonian}

In this section we define CRNs and their associated Hamiltonians for the case
of mass-action kinetics. \

\subsection{Reaction networks}

\begin{definition}
  Let $\mathcal{A}= \{ A_1, \ldots, A_d \}$, let $\mathcal{C}$ be a finite
  subset of $\mathbb{Z}^d_+$, and let $\mathcal{R} \subset \mathcal{C} \times
  \mathcal{C}$ be a subset of ordered pairs of complexes. To each ordered pair
  $(a, b) \in \mathcal{R}$ we associate the vector $\zeta_{a, b} \doteq b - a
  \in \mathbb{Z}^d$. The elements $A_i \in \mathcal{A}$ are called
  {\emph{species}}, the $a \in \mathcal{C}$ are called {\emph{complexes}}, the
  vectors $\zeta_{a, b}$ are called {\emph{reaction vectors}}, and the tuple
  $G = (\mathcal{A}, \mathcal{C}, \mathcal{R})$ is called a {\emph{chemical
  reaction network (CRN)}}.
\end{definition}

\begin{definition}\label{def:domain}
  Let $G = (\mathcal{A}, \mathcal{C}, \mathcal{R})$ be a CRN. The subspace
  \[ \mathcal{S}_0 = \left\{ \sum_{(a, b) \in \mathcal{R}} c_{a, b} \zeta_{a,
     b} : c_{a, b} \in \mathbb{R}, (a, b) \in \mathcal{R} \right\} \]
  is called the {\emph{stoichiometric subspace}} of the network. For a given
  initial condition $x (0) \in \mathbb{R}^d_{\geqslant 0}$, we let
  \[ \mathcal{S}_{x (0)} \doteq \mathbb{R}^d_{\geqslant 0} \cap
     (\mathcal{S}_0 + x (0)) =\mathbb{R}^d_{\geqslant 0} \cap \{ s + x (0) : s
     \in \mathcal{S}_0 \}, \]
     and we refer to this as the stoichiometric compatibility class passing through $x(0)$, or simply as the domain of the network.  When the particular initial condition $x (0)$ is irrelevant, or when
  $\mathcal{S}_x$ is the same for all choices of $x$, we will denote the domain by $\mathcal{S}$.
\end{definition}

We will use the following notation throughout.

\begin{convention}
  We use the convention of exponential notation for $x \in \mathbb{R}^d_+$ and
  $a \in \mathbb{Z}^d_+$
  \[ x^a = x_1^{a_1} \cdots x_d^{a_d} . \]
\end{convention}
\begin{remark}
We will strive to use $x$ to denote a point in $\mathcal{S} \subset \mathbb{R}^d$ and $x(\cdot)$ to denote a path $x(\cdot) \doteq \{x(t), t\in[0, T]\}$. We do not call this a convention, because we may not have been stringent enough in enforcing it, and hopefully no confusion will arise. 
\end{remark}

To a CRN $G = (\mathcal{A}, \mathcal{C}, \mathcal{R})$ we associate \emph{rate constants} $\kappa = \{ \kappa_{a, b} > 0 : (a, b) \in \mathcal{R} \}$ which are necessary to define the stochastic and deterministic dynamics. Given a reaction network $G$ and rate constants $\kappa$, we 
associate an object we will call the \emph{Hamiltonian} of the reaction
network. This is an object which is central in providing the link between the
stochastic and deterministic dynamics.

\begin{definition}
  Let $G = (\mathcal{A}, \mathcal{C}, \mathcal{R})$ be a CRN with associated
  rate constants $\kappa = \{ \kappa_{a, b} > 0 : (a, b) \in \mathcal{R} \}$.
  We define the {\emph{Hamiltonian}} associated to $G$ by
  \begin{equation}
    H_G (x, p) \doteq \sum_{(a, b) \in \mathcal{R}} \kappa_{a, b} x^a 
    (e^{\langle \zeta_{a, b}, p \rangle} - 1) . \label{eqn:hamiltonian}
  \end{equation}
  The set of all functions $H : \mathbb{R}^d_+ \times \mathbb{R}^d \rightarrow
  \mathbb{R}$ of the form {\eqref{eqn:hamiltonian}} with $\kappa_{a, b} > 0$,
  $a \in \mathbb{Z}^d_+$ and $\zeta_{a, b} \in \mathbb{Z}^d$ is in one-to-one
  correspondence with the set of CRNs with mass-action kinetics. \ 
\end{definition}

\begin{example}
  \label{ex:guiding_example}Consider the one-dimensional reaction network \
  with a single species $\mathcal{A}= \{ A \}$, two complexes $\mathcal{C}= \{
  0, 1 \}$ and two reactions $\mathcal{R}= \{ (0, 1), (1, 0) \}$. To this CRN
  we associate the rate constants $\kappa_{0, 1} = \kappa_{1, 0} = 1$. This
  corresponds to the reaction network
  \[ A \rightleftarrows \emptyset \]
  where both rate constants are equal to 1. The associated Hamiltonian is
  \[ H (x, p) = x (e^{- p} - 1) + e^p - 1. \]
\end{example}

We will refer to Example \ref{ex:guiding_example} to illustrate various objects introduced in this note.

\section{From CRNs to
Hamilton-Jacobi-Bellman}\label{sec:ldp_to_classical_mechanics}

In this section we sketch the connection between CRNs and the
Hamilton-Jacobi-Bellman equation by appealing to large deviations theory. We begin in Section \ref{sec:roadmap} by providing a roadmap for the reader which identifies the important definitions and the equations relating them. The key idea is that in the rare event regime, the qualitative prediction of an event is governed asymptotically by its cost. The difficulty is that evaluating the cost of an event is in general a non-convex infinite-dimensional optimization problem. 

The remaining subsections provide greater detail on parts of the roadmap. Section \ref{sec:crn_to_lln} focuses on the law of large numbers / deterministic dynamics in terms of the Hamiltonian; Section \ref{sec:lln_to_ldp} provides a brief overview of large deviations theory, and \ref{sec:ldp_to_hjb} explores the calculus of variations and HJB approaches to the optimization problem. 

In the end we cover only the very basic concepts and omit many
important technical details, such as the existence and uniqueness of
minimizers of variational problems and the notion of spatial derivatives on
the boundary. The interested reader can consult \cite{Varadhan08} or \cite{DupuisPadova} for an introduction to large deviations. 

\

\subsection{Roadmap}\label{sec:roadmap}

\


We associate to any CRN $G$ a sequence of pure jump Markov processes $\{ X^n,
n \geqslant 1 \}$ defined on a fixed time interval $[0, T]$. Each process $\{
X^n (t), t \in [0, T] \}$ is an element of the space of right-continuous
trajectories with left limits known as Skorokhod space, which we denote $D_T
\doteq D ([0, T] : \mathcal{S})$. It will turn out that under appropriate
conditions, the sequence $\{ X^n, n \geqslant 1 \}$ satisfies a Large
Deviations Principle on $D_T$. Roughly speaking, such a ``path-space'' LDP  for $\{X^n\}$ implies that there is a
function $I_T : D_T \rightarrow \infty$ such that
\begin{equation}
  \mathbb{P} (X^n \in A) \approx \inf_{\varphi \in A} e^{- n I_T (\varphi)} .
  \label{eqn:ldp_prob_estimate}
\end{equation}
for ``nice enough'' $A \subset D_T$. In other words, large
deviations provides a relationship between the \emph{probability}
$\mathbb{P} (X^n \in A)$ of an event $A$ and its \emph{cost}
$\inf_{\varphi \in A} I_T (\varphi)$. The more ``expensive'' an event, the
less likely it is. In the case of CRNs, the function $I_T$ will be of the form
of an \emph{action functional},
\[ I_T  (\varphi) = \int_0^T L (\varphi (t), \dot{\varphi} (t)) \text{dt}, \]
where $L (x, \beta)$ is real-valued and non-negative, defined for $(x, \beta)
\in \mathbb{R}^d_+ \times \mathbb{R}^d$, possibly taking the value $+ \infty$,
and $L (x, \cdot)$ is convex for each $x \in \mathbb{R}^d_{> 0}$. The function
$L$ will be called the \emph{local rate} or \emph{Lagrangian}, and the integral of $L$ over a trajectory will be called the $I_T (\varphi)$ the \emph{cost} of a trajectory. Under general conditions (which hold for CRNs), if $I_T(\varphi) = 0 $ then $\varphi$ satisfies the deterministic dynamics.

It is in general difficult to evaluate $\inf_{\varphi \in A} I_T (\varphi)$
for general $A \subset D_T$. A well-studied class of events $A \subset D_T$ is
the set of trajectories satisfying \emph{finite time endpoint
constraints}: for any $x, y \in \mathcal{S}$ and $T > 0$, we consider
\begin{equation}
  A_T (x, y) = \{ \varphi \in D_T : \varphi (0) = x, \varphi (T) = y \} .
  \label{eqn:A_T}
\end{equation}
The problem of estimating the exponential decay rate of the probability
$\mathbb{P} (X^n \in A_T (x, y))$ is reduced by the large deviations relation
in {\eqref{eqn:ldp_prob_estimate}} to solving the \emph{calculus of
variations} problem
\[ \inf_{\varphi} \left\{ \int_0^T L (\varphi (t), \dot{\varphi} (t))
   \text{dt} : \varphi (0) = x, \varphi (T) = y \right\} .
   \label{eqn:calculus_of_variations_problem} \]
A standard way to identify the minimizers of a calculus of variations problem
is to consider the \emph{Euler-Lagrange} equation
\[ \frac{d}{d t} \left( \frac{\partial L}{\partial \beta} (\varphi (t),
   \dot{\varphi} (t)) \right) = \frac{\partial L}{\partial x}, \]
with boundary conditions $\varphi (0) = x, \varphi (T) = y$. One can express
the Euler-Lagrange equation in the ``Hamiltonian formalism'' by introducing
the \emph{Hamiltonian}
\begin{equation}
  H (x, p) = \sup_{\beta \in \mathbb{R}^d} \{ \langle p, \beta \rangle - L (x,
  \beta) \} . \label{eqn:hamiltonian_legendre}
\end{equation}
Formally, by introducing the variable $p (x, \beta) = \partial L (x, \beta) /
\partial \beta$ and separating the second order equation into a pair of two
first order equations, we obtain \emph{Hamilton's equations of motion}
\begin{eqnarray}
  \dot{x} (t) & = & D_p H (x (t), p (t)), 
  \label{eqn:equation_of_motion_1}\\
  \dot{p} (t) & = & -D_x H (x (t), p (t)), 
  \label{eqn:equation_of_motion_2}
\end{eqnarray}
with boundary conditions given by $x (0) = x$ and $x (T) = y$. If $(x (\cdot),
p (\cdot))$ is a pair of solutions to the equations of motion
{\eqref{eqn:equation_of_motion_1}}-{\eqref{eqn:equation_of_motion_2}},
then $\{ x (t) : t \in [0, T] \}$ is (under suitable conditions) a trajectory
which achieves the minimum in the variational problem $\mathcal{J}_T (x, y)$. (For further reading, consider p. 255, Section D: \emph{The Hamilton-Jacobi equation} in \cite{Arnold}.)

Observe that while a solution $(x (\cdot), p (\cdot))$ to the equation of motion can in
principle be used to evaluate the cost $\inf_{\varphi} I_T (A_T (x, y))$ by
evaluating $I_T (x (\cdot))$, it provides more information than needed for the
task of evaluating the decay rate of the event, as in
{\eqref{eqn:ldp_prob_estimate}}. For that, one only needs \emph{the}
\emph{minimizing value}, not \emph{a} \emph{minimizer}.

Consider the idea of parametrizing the cost $\mathcal{J}_T (x_0, x)$ for some
fixed initial position $x_0$ in terms of the terminal condition $x \in
\mathbb{R}^d_{\geqslant 0}$ and the terminal time $T > 0$. \ We obtain a
function $V (T, x) = \mathcal{J}_T(x_0, x)$ which formally solves the Hamilton-Jacobi-Bellman (HJB)
partial differential equation (PDE)
\begin{equation}\label{eqn:hjb_plus} 
\partial_T V (T, x) + H (y, D_x V (T, x)) = 0, 
\end{equation}
with boundary condition $V (0, x) = 0$ if $x = x_0$ and $V (0, x) = + \infty$
otherwise. The function $V (T, x)$ represents \ ``the cost to move from $x_0$ to
$x$ in time $T$''. Under additional conditions on $x_0$ which will be discussed in Section \ref{sec:quasipotential}, if one instead solves the time-independent (or
\emph{steady state}) HJB,
\[ H (x, D_x V (x)) = 0, \]
with boundary condition $V (x) = 0$ for $x = x_0$, then the solution $V (x)$
represents ``the cost to move from $x_0$ to $x$ in an arbitrarily long but
finite amount of time''. \\

To summarize, 
\begin{enumerate}
\item the exponential decay rate of the probability of an event with LLN scaling is governed by the cost of the event;
\item evaluating the cost of an event can be seen as a calculus of variations problem with constraints;
\item for the event of moving between one point and another in a fixed amount of time there is a PDE for the cost;
\item the steady-state form of the PDE corresponds to the cost of an event over an ``arbitrarily long but finite'' time horizon. 
\end{enumerate}

\subsection{Stochastic dynamics and its Law of Large Lumbers
}\label{sec:crn_to_lln}

We now provide the details for the steps outlined at the beginning of the
section, starting with a more precise definition of the stochastic dynamics
and the associated law of large numbers.

\

Let $G = (\mathcal{A}, \mathcal{C}, \mathcal{R})$ be a $d$-dimensional CRN
with associated rate constants $\kappa = \{ \kappa_{a, b} : (a, b) \in
\mathcal{R} \}$, and choose any $n \geqslant 1$. This will be the scaling
parameter used for a law of large numbers. It is intended to represent the
volume of the container. For $x \in n^{- 1} \mathbb{Z}^d_+$, we define
\[ \lambda^n_{a, b} (x) = \frac{\kappa_{a, b}}{n^{| a |}}  \prod_{i = 1}^d
   \frac{(n x_i) !}{(n x_i - a_i) !} 1 \{ n x_i \geqslant a_i \} \]
and
\[ \lambda_{a, b} (x) = \kappa_{a, b} x^a . \]
Note that if $x \in \mathbb{R}^d_{\geqslant 0}$ and $x_n \in n^{- 1}
\mathbb{Z}^d_+$, $x_n \rightarrow x$, then $\lambda^n_{a, b} (x_n) \rightarrow
\lambda_{a, b} (x)$.

\

For any fixed $T > 0$, we associate to $G$ a sequence of jump Markov
processes $\{ X^n (t), t \in [0, T] \}_{n \geqslant 1}$ with generator
\[ \mathcal{L}^n f (x) = \sum_{(a, b) \in \mathcal{R}} n \lambda^n_{a, b}
   (x) (f (x + \zeta_{a, b} / n) - f (x)) . \]
We assume that $X^n (0) \in n^{- 1} \mathbb{Z}^d$, and consequently $X^n$ is a
$D ([0, T] : n^{- 1} \mathbb{Z}^d)$-valued random variable. The family of
generators associated to the CRN in Example \ref{ex:guiding_example} is
\[ \mathcal{L}^n f (x) = n x 1_{\{ n x \geqslant 1 \}} \left( f \left( x -
   \frac{1}{n} \right) - f (x) \right) + n \left[ f \left( x + \frac{1}{n}
   \right) - f (x) \right] . \label{ex:generator} \]

It turns out that $X^n$ satisfies a Law of Large Numbers (LLN), that is, there
is a function $x (\cdot) : [0, T] \rightarrow \mathcal{S}$ such that, for any
$\delta > 0$,
\begin{equation}
  \mathbb{P} \left(\sup_{0 \leqslant t \leqslant T} | X^n (t) - x (t) | > \delta\right)
  \xrightarrow{n \rightarrow \infty} 0. \label{eqn:lln_probability}
\end{equation}
We refer to the limit in {\eqref{eqn:lln_probability}} as Kurtz's theorem
{\cite{KurtzLLN77}}{\cite{KurtzCRN72}}. One can formally guess the LLN limit
$x (\cdot)$ by the heuristic calculation
\[ \lambda_{a, b}^n (x)  \left( \frac{f (x + \zeta_{a, b} / n) - f (x)}{1 / n}
   \right) \approx \lambda_{a, b} (x)  \langle D f (x), \zeta_{a, b}
   \rangle, \]
for $f \in C^1 (\mathbb{R}_+) .$ It turns out that $x (\cdot)$ is absolutely
continuous and satisfies $\dot{x} (t) = D_p H (x (t), 0)$, where
\begin{equation}
  D_p H (x (t), 0) \doteq \sum_{(a, b) \in \mathcal{R}} \lambda_{a, b}
  (x (t)) \zeta_{a, b} . \label{eqn:deterministic_dynamics}
\end{equation}
We call {\eqref{eqn:deterministic_dynamics}} the \emph{deterministic
dynamics} of a CRN. For the network in Example \ref{ex:guiding_example}, we
find
\[ D_p H (x, 0) = 1 - x. \]
\subsection{Large devations: LLN's rate of convergence}\label{sec:lln_to_ldp}

Whenever one has a sequence of random variable satisfying a law of large
numbers, one can ask whether the sequence satisfies a large deviations
principle. Roughly speaking, a large deviations principle identifies the
\emph{rate of decay} in $n$ of the probabilities
{\eqref{eqn:lln_probability}}.

In this section we define rate functions and a large deviations principle in
abstract terms. We let $\mathcal{X}$ denote a complete, seperable metric
space, also known as a Polish space.

\begin{definition}
  A function $I : \mathcal{X} \rightarrow [0, \infty)$ is a {\emph{rate
  function}} if for each $M < \infty$, the sublevel set $\{ x \in \mathcal{X}:
  I (x) \leqslant M \}$ is compact. 
\end{definition}

We recall that a having {\emph{closed}} sublevel sets is equivalent to being
lower semicontinuous, and that any lower semicontinuous function achieves its
minimum on a compact set. It follows that a rate function achieves its minimum
on any sublevel set.

\begin{definition}
  A sequence $Y^n$ satisfies a {\emph{Large Deviations Principle (LDP)}} with
  rate function $I$ and speeed $n^{- 1}$ if, for any open set $E \subset
  \mathcal{X}$, we have the large deviations lower bound
  \[ \liminf_{n \rightarrow \infty} \frac{1}{n} \log \mathbb{P} (Y^n \in E)
     \geqslant - \inf_{y \in E} I (y) \]
  and for any closed set $F \subset \mathcal{X}$, we have the large deviations
  upper bound
  \[ \limsup_{n \rightarrow \infty} \frac{1}{n} \log \mathbb{P} (Y^n \in F)
     \leqslant - \inf_{y \in F} I (y) . \]
  We define, for any set $A \subset \mathcal{X}$,
  \[ I (A) \doteq \inf_{y \in A} I (y), \]
  and we will say that an event is {\emph{rare}} if $I (A) > 0$. 
\end{definition}

\subsubsection{Large deviations for CRNs}

There are at least two obstacles when trying to establish a large deviations
principle for chemical reaction networks. First, one must place appropriate
assumptions on the jump intensities to ensure that they do not diverge ``too
quickly'' and that there is an appropriate controllability condition; see
{\cite{Agazzi18}} for such a result. Another important issue is boundary
behavior: one needs to account for the possibility of species vanishing. A
proof of the LDP which accounts for this kind of behavior has recently been proposed by
{\cite{RengerLDP18}}.

The result is likely to hold on the space $D ([0, T])$ with the usual
Skorokhod topology under milder conditions than what has been established as
of yet. If the result does hold, then one can be confident the rate function
will be
\[ I_T (\varphi) = \left\{\begin{array}{ll}
     \int_0^T L (\varphi (s), \dot{\varphi} (s)) \text{ ds } & \text{ if } \varphi
     \text{ is } \text{AC}\\
     + \infty & \text{else}
   \end{array}\right., \]
where for $(x, \beta) \in \mathbb{R}^d_+ \times \mathbb{R}^d$,
\[ L (x, \beta) = \inf_{u_{a, b}} \left\{ \sum_{(a, b) \in \mathcal{R}}
   \lambda_{a, b} (x) \ell \left( \frac{u_{a, b}}{\lambda_{a, b} (x)} \right)
   : u_{a, b} \geqslant 0 \text{ for } (a, b) \in \mathcal{R}, \beta = \sum_{(a,
   b) \in \mathcal{R}} \zeta_{a, b} u_{a, b} \right\}, \]
and
\[ \ell (x) \doteq x \log x - x + 1, \]
with $\ell (0) = 0$ by convention. We call $I_T (\varphi)$ the \emph{cost}
of the trajectory $\varphi$. An important regularity property of $L$ is that for any $x \in \mathcal{S}$, there is a unique $b(x)$ such that $L(x, b(x)) = 0$, and moreover, the function $b(x)$ is Lipschitz continuous. Indeed, $L (x, \beta) = 0$ if and only if
$u_{a, b} = \lambda_{a, b} (x)$ for all $(a, b) \in \mathcal{R}$, or
equivalently $\beta = \sum_{(a, b) \in \mathcal{R}} \zeta_{a, b} \lambda_{a,
b} (x) = D_p H (x, 0)$. 

Moreover, since $L$ is nonnegative, it must be that any
$\varphi$ which satisfies $I_T (\varphi) = 0$ satisfies the deterministic
dynamics, which are sometimes referred to as the \emph{zero-cost}
dynamics. We remark that the rate function $I_T(\cdot)$ depends on the time interval $[0, T]$ on which the trajectories are defined. 

\subsubsection{Heuristics for large deviations}

Rigorously justifying the form of the large deviations rate function is beyond
the scope of this paper, and we refer the reader interested in the details to
an introductory textbook such as {\cite{SchwartzWeiss95}}. In the interest in
providing some motivation for the form of $L$ and $H$, we provide two examples
with heuristic calculations.

\begin{example}
  \label{ex:cramer}Here is a basic calculation that motivates the definition
  of the Hamiltonian and the Legendre transform. Suppose $\{ \xi_i \}_{i =
  1}^{\infty}$ are i.i.d. random variables with ``light tails''. Then, for
  $\beta > 0$ and any $p \geqslant 0$, we have by Chebyshev's inequality that
  \[ \mathbb{P} \left( \frac{1}{n} \sum_{i = 1}^n \xi_i \geqslant \beta
     \right) =\mathbb{P} \left( e^{p \sum_i \xi_i} \geqslant e^{p n \beta}
     \right) \leqslant \frac{\mathbb{E} \left[ e^{p \sum_i \xi_i}
     \right]}{e^{p n \beta}} = e^{- n (p \beta - H (p))}, \]
  where
  \[ H (p) = \log \mathbb{E} [e^{p \xi_1}] . \]
  We can optimize over the parameter $p$ by taking the infimum on the
  right-hand side with the inequality remaining valid. It can be shown under
  appropriate conditions that, in fact,
  \[ \lim_{n \rightarrow \infty} \frac{1}{n} \log \mathbb{P} \left(
     \frac{1}{n} \sum_{i = 1}^n \xi_i \geqslant \beta \right) = L (\beta) \]
  where $L (\beta) = \sup_{p \in \mathbb{R}} \{ p \beta - H (p) \} .$ This is
  a result covered by what is known as Cram\'er's theorem; see \cite{DemboBook}, p. 26. 
\end{example}

The preceding example allows us to guess the rate function on path space. 
\begin{example}
Consider a stochastic process of the form 
\begin{equation}
  Y^n \left( \frac{i + 1}{n} \right) = Y^n \left( \frac{i}{n} \right) +
  \frac{1}{n} \xi_i, \label{eqn:process}
\end{equation}
where the $\xi_i$ are i.i.d. random variables with ``light tails''. The scaling in {\eqref{eqn:process}} is a LLN-type scaling, and it is
reasonable to expect that
\[ Y^n (t) \xrightarrow{n \rightarrow \infty} t\mathbb{E}X_1, \quad t \in [0,
   1] . \]
The key here is a \emph{time-scale separation} due to
the $1 / n$ scaling. Over a time interval $[t, t + \delta]$, $Y^n (t)$ will
change at most $O (\delta)$ with high probability, since the increments $\xi_i$  are light-tailed (say, sub-Gaussian). For
$\delta > 0$,
\begin{equation}
  \frac{Y^n (t + \delta) - Y^n (t)}{\delta} \approx \frac{1}{n \delta} \sum_{i
  = \lfloor n t \rfloor}^{\lfloor n t + n \delta \rfloor} \xi_i
  \label{eqn:time-scale} .
\end{equation}
For large $n$ and $\delta > 0$ small, we can think of the left-hand side of
{\eqref{eqn:time-scale}} as a sum of $\lfloor n \delta\rfloor$ i.i.d. random
variables. By Cram\'er's theorem,
\begin{equation}
  \mathbb{P} \left( \frac{Y^n (t + \delta) - Y^n (t)}{\delta} \approx \beta
  \right) \approx e^{- n \delta L (\beta)}, \label{eqn:approx2}
\end{equation}
where $L$ is as in Example \ref{ex:cramer}. The random variables in this example are trajectories, and $\beta$ can be interpreted as the
\emph{velocity} of $Y^n$. Heuristically, the probability of $Y^n$ being close to an (absolutely continuous) path $\{\varphi(t), t \in [0, 1]\}$ is, for small $\delta > 0$ an, roughly
\begin{eqnarray*}
  \mathbb{P} (Y^n \approx \varphi) & = & \mathbb{P} (\| Y^n - \varphi
  \|_{\infty} \leqslant \varepsilon)\\
  & \approx & \mathbb{P} (Y^n (j \delta) \approx \varphi (j \delta), j = 0,
  \ldots, \lfloor 1 / \delta \rfloor)\\
  & \approx & \mathbb{P} \left( \frac{Y^n (j \delta + \delta) - Y^n (j
  \delta)}{\delta} \approx \frac{\varphi (j \delta + \delta) - \varphi (j
  \delta)}{\delta}, j = 0, \ldots, \lfloor 1 / \delta \rfloor \right)\\
  & \approx & \prod_{j = 0}^{\lfloor 1 / \delta \rfloor} e^{- n \delta L
  \left( \frac{\varphi (j \delta + \delta) - \varphi (j \delta)}{\delta}
  \right)}\\
  & \approx & e^{- n \int_0^1 L (\dot{\varphi} (s)) \text{ds}} .
\end{eqnarray*}
In the above we used {\eqref{eqn:approx2}} with (essentially) $\beta =
\dot{\varphi}$, and we crucially used independence of the increments to break
up the intersection of events into a product. This example shows one way of guessing the form of the rate function for processes. 
\end{example}

\begin{remark}
If the distribution of the increment $\xi_n$ in the preceding example is taken to be dependent on the previous state $Y_{n-1}$, then under appropriate assumptions the local rate $L$ will be state-dependent as well, and $L(x,\beta)$ can be defined for fixed $x$ by the duality formula {\eqref{eqn:hamiltonian_legendre}},
\begin{equation}\label{eqn:legendre_transform}
 L (x, \beta) = \sup_{p \in \mathbb{R}^d} \{ \langle p, \beta \rangle - H
   (x, p) \}. 
   \end{equation}
\end{remark}
\subsection{From large deviations to
Hamilton-Jacobi-Bellman}\label{sec:ldp_to_hjb}
While a large deviations principle allows one to identify the asymptotic decay rate of any sufficiently nice event $A \subset \mathcal{X}$, it is in general difficult to solve the optimization problem $\inf_{\varphi \in A} I_T(A)$. In this section we elaborate on the details of taking a PDE perspective on the optimization problem for a particular class of events.  The underlying PDE is called the Hamilton-Jacobi-Bellman (HJB) equation, and is sometimes referred to as the \emph{dynamic programming} equation, especially in discrete space and time. The techniques in continuous space and time trace back to Lagrangian and Hamiltonian classical mechanics, as well as optimal control and the calculus of variations. 

We first outline the main ideas. As in Section \ref{sec:roadmap}, consider the family of events parametrized by the initial and terminal conditions,
\[ A_T (x, y) = \{ \varphi \in D ([0, T]) : \varphi (0) = x, \varphi (T) = y
   \} . \]
We let
\begin{equation}
  \mathcal{J}_T (x, y) = \inf_{\varphi \in A_T (x, y)} I_T (\varphi) .
  \label{eqn:J_T}
\end{equation}
Thus, the probability of going from $x \in \mathbb{R}^d$ to $y \in
\mathbb{R}^d$ in time $T > 0$ is roughly
\[ \mathbb{P}_x (X^n (T) \in B_{\delta} (y)) \approx e^{- n\mathcal{J}_T (x,
   y)} . \]
The problem of evaluating $\mathcal{J}_T (x, y)$ is a finite time problem
with endpoint constraints. Roughly speaking, $\mathcal{J}_T (x, y)$ measures
the cost of the stochastic process $X^n$ moving from $x$ to $y$ in time $T$.
The most likely trajectory for $X^n$ to follow starting at $X^n (0) = x$ is
the law of large numbers dynamics, $\dot{x} (t) = D_p H (x (t), 0)$ with
$x (0) = x$, and in this case we expect $X^n (T) \approx x (T)$. The cost of
this trajectory is identically zero: $\mathcal{J}_T (x (0), x (T)) = 0$.

We now sketch the connection of $\mathcal{J}_T(x, y)$ to the HJB PDE. For some fixed $y \in \mathbb{R}^d_{> 0}$ and
all $x \in \mathcal{S}_y$, $t \in [0, T]$, define
\begin{align}\label{eqn:value_function_defn} 
V (t, x) &= \inf_{\varphi} \left\{ \int_t^T L (\varphi (s), \dot{\varphi}
   (s)) \text{ds} : \varphi (t) = x, \varphi (T) = y \right\} \\
   &=\mathcal{J}_{T   - t} (x, y),\nonumber
\end{align}
   and define $V(t, x) = +\infty$ if the set over which the infimum is taken is empty. 
   
   The connection to the PDE is made by observing that if $\varphi$ is an optimal trajectory starting at $(t, x)$ and ending at $(T, y)$, and $\varphi$ passes through $(t+\delta, \varphi(t+\delta))$, then $\varphi$ should also be optimal to move from $(t+\delta, \varphi(t+\delta))$ to $(T, y)$. This is known as the ``principle of optimality''. In other words, we expect
\begin{align*}
V(t, x) &= \min_{\varphi} \bigg\{V(t+\delta, \varphi(t+\delta)) \\
&\quad+\int_t^{t+\delta}L(\varphi(s), \dot\varphi(s)) ds : \varphi(t) = x, \varphi(T) = y \bigg\}.
\end{align*}
By subtracting $V(t, x)$ from both sides, dividing by $\delta$, sending $\delta \rightarrow 0$ and using the Fenchel-Legendre duality \eqref{eqn:legendre_transform}  with $\dot{\varphi}(t)$ in place of $\beta$ and $-D_x V(t, x)$ in place of $p$, we find that $V$ formally satisfies the HJB
\begin{equation}\label{eqn:pde_1}
 \partial_t V (t, x) - H (x, -D_x V (t, x)) = 0 
 \end{equation}
with boundary condition $V (T, x) = 0$ if $x = y$ and $V (T, x) = + \infty$
otherwise. Note that the PDE \eqref{eqn:pde_1} differs by two negative signs from the equation in \eqref{eqn:hjb_plus}, due to the fact that in the present case the terminal time and position are fixed, while in the other case the terminal time and position are the variables. The two choices represent fundamentally different, but related problems.

A precise statement relating the variational definition of $V$ in \eqref{eqn:value_function_defn} to the solution of \eqref{eqn:pde_1}  is difficult since the very notion of solution for HJB equations is itself quite subtle. For instance, the function $V$ need not be differentiable. Crandall and Lions established the notion of \emph{viscosity} solutions to resolve this issue in their seminal work \cite{Crandall83}. The theory is technical and outside the scope (or purpose) of this note, but see:  \cite{Feng06}, \cite{Bardi97}, Theorem 5, p. 128 in \cite{Evans98} for a simplified situation, or \cite{Nyquist14} for more general extensions in the ``infinite horizon'' case (which will be discussed below). We will primarily use the results in \cite{FreidlinWentzell} Chapters 3-5. 

The HJB PDE provides another approach to establishing a LDP (this is outlined in \cite{Feng06}). It is also useful for the design of rare event simulation schemes such as importance sampling or splitting (see \cite{DupuisWang07, DupuisKostas15, DeanDupuis09, DeanDupuis11}). Certain  kinds of rare events, such as situations where a stochastic process ``jumps'' between two metastable states, may occur over very long timescales. In fact, one can prove that under appropriate assumptions on a domain $D$ containing a unique stable steady state, and an appropriately scaled stochastic process $X^n$ satisfying an appropriate ``uniform'' LDP, we have 
\[ \lim_{n\rightarrow\infty} -\frac{1}{n}\log \mathbb{P}_x(X^n(t) \notin D, \text{ for some } t \leq n^k) = V(x),\quad x\in D, \]
for any $k \geq 1$, where $V (x)$ is given by the variational problem
\[ V (x) = \inf_{T < \infty} \inf_{y \notin D} \mathcal{J}_T (x, y). \]
See \cite{Snarski19} for a proof and an application to splitting methods. The function $V (x)$ formally satisfies the \emph{steady-state} HJB equation,
\[ H (x, -D_x V (x)) = 0, \]
with boundary condition $V (x) = 0$ for $x \in \partial D$. The variational formulation $V(x) = \inf_{T < \infty} \inf_{y \notin D} \mathcal{J}_T (x, y)$ allows us to interpret $V(x)$ as being the cost to exit $D$ starting from $x \in D$. 

There is a related problem which considers the cost of hitting a point $x \in D$ starting from the steady state in an arbitrarily long but finite amount of time. Given an appropriate domain $D\subset\mathbb{R}^d$ with a unique stable steady state $c \in D$, the solution $V(x)$ of 
\[ H(x, DQ(x)) = 0 \]
with $Q(c) = 0$ identifies the cost of moving from $c$ to $x\in \{ y : Q(y) \leq \min_{z \in \partial D} Q(z)\}$; see Theorem 4.3 in Chapter 5 of \cite{FreidlinWentzell}.  The solution $Q$ is a classical object known as the Freidlin-Wentzell \emph{quasipotential} and is the central object of study in this note; see the next section. 

We end the section with a conjecture that is related to the ideas discussed in this section but has no relation to the rest of the note. 

\paragraph{A conjecture.}
This conjecture is related to necessary conditions for a trajectory to be a minimizer of the optimization problem $\mathcal{J}_T(x, y)$.   claim that a ``Lagrange multiplier'' approach to the optimization problem $\mathcal{J}_T(x, y)$. The infinite-dimensional analogue of the Karush-Kuhn-Tucker conditions which provide necessary conditions for Lagrange multipliers is the celebrated \emph{Pontryagin Maximum Principle}; see \cite{PontryaginBook}, also \cite{Vinter04} for the same result under weaker conditions. 

We conjecture that if the pair $(x (t), p (t))$ satisfies
\begin{eqnarray*}
  \dot{x} (t) & = & \sum_{(a, b) \in \mathcal{R}} \zeta_{a, b} \lambda_{a, b}
  (x (t)) e^{\langle p (t), \zeta_{a, b} \rangle}\\
  \dot{p} (t) & = & \sum_{(a, b) \in \mathcal{R}} D_x \lambda_{a, b} (x
  (t))  (1 - e^{\langle p (t), \zeta_{a, b} \rangle}),
\end{eqnarray*}
with boundary conditions $x (0) = x$, $x (T) = y$, then $I_T (x (\cdot))
=\mathcal{J}_T (x, y)$.
One cannot apply even the weakest forms of the Pontryagin Maximum Principle since the underlying class of controls does not satisfy the integrability assumptions. This lack of integrability is related to the blowup of $L (x, \beta)$ as $x \rightarrow
\partial \mathcal{S}$ for certain directions $\beta$. We also note that for $x \in
\partial S$ the spatial gradient must be given an appropriate interpretation. 

\section{Quasipotential, Lyapunov functions, and
HJB\label{sec:quasipotential}}

\

We consider a family of variational problems which define an object known as
the \emph{quasipotential}. The quasipotential was first introduced by
Freidlin and Wentzell and is often used to study the long-term dynamics of
stochastic systems.

There are three important results we wish to communicate: the relationship of
the quasipotential to Lyapunov functions, invariant measures, and the
Hamilton-Jacobi-Bellman equation. In Section \ref{sec:quasi_defn}, we define
the quasipotential, as well as the notion of a Lyapunov function, and we show
in Section \ref{sec:quasipotential_lyapunov} that the quasipotential serves as
a Lyapunov function for the noiseless dynamics. In Section
\ref{sec:quasipotential_invariant} we demonstrate how the quasipotential can be deduced
by an appropiately scaled limit of invariant measures for the process $X^n$,
and in Section \ref{sec:quasipotential_hjb} we show how it can instead by
obtained from the HJB PDE.

\subsection{Definitions}\label{sec:quasi_defn}

Recall that $\mathcal{J}_T (x, y)$ denotes the optimal value of a variational
problem, as defined in {\eqref{eqn:J_T}}.

\begin{definition}
  The quasipotential between two points $x, y \in \mathbb{R}^d$ is defined as
  \[ Q (x, y) = \inf_{T < \infty} \mathcal{J}_T (x, y), \]
  with $Q (x, y) = + \infty$ if $y \notin \mathcal{S}_x$ (recall Definition \ref{def:domain}). 
\end{definition}

Thus the quasipotential in some sense measures the ``optimal cost'' of moving
from $x$ to $y$ when allowing for an arbitrarily long but finite amount of
time. Note that the infimum need not be achieved, or if it is achieved the
minimizer need not be unique. \\

Next, we define more precisely a term that has been used throughout the note. 
\begin{definition}[Stable steady state]
A steady state $c$ of the dynamical system $\dot{x}(t) = \Phi(x(t))$ is said to be \emph{stable} if for every neighbourhood $U$ of $c$ there is a smaller neighbourhood $V \subset U$ such that the trajectories of $\dot{x}(t) = \Phi(x(t))$ starting in $V$ converge to $c$ without leaving $U$. We say that a domain $D$ containing $c$ is attracted to $c$ if the trajectories $\dot{x}(t) = \Phi(x(t))$ converge to $c$ without leaving $D$.
\end{definition}

We will assume throughout without further mention that $D$ will denote a domain that is attracted to a stable steady state $c$ under the deterministic dynamics $\Phi(x) = D_p H(x, 0)$, and we refer to the maximal open set of points that are attracted to $c$ as the \emph{domain of attraction}. \\

The study of stability is aided by the notion of a Lyapunov function. 

\begin{definition}[Lyapunov function]
  \label{def:Lyapunov_function}Let $D \subset \mathbb{R}^d_+$ be an open
  subset of $\mathbb{R}^d_+$ and let $\Phi : \mathbb{R}^d_+ \rightarrow
  \mathbb{R}^d$. Suppose that $c \in \mathbb{R}^d_+$ is a stable steady state
  for the system
  \[ \dot{x} (t) = \Phi (x (t)) . \]
  A continuously differentiable function $U : D \rightarrow \mathbb{R}$ is
  called a {\emph{strict Lyapunov function}} for the system $\dot{x} (t) =
  \Phi (x (t))$ at $c$ if $U (c) = 0$, $U (x) > 0$ for all $x \neq c$, and
  \[ \langle D U (x), \Phi (x) \rangle \leqslant 0 \]
  for all $x \in D$, with equality if and only if $x = c$.
\end{definition}

\subsection{Quasipotential and Lyapunov
functions}\label{sec:quasipotential_lyapunov}

In this section we show that the quasipotential serves as a Lyapunov function
for the deterministic dynamics. This can be understood heuristically as
follows. The quasipotential $Q (x)$ represents the minimal ``cost'' to deviate
from a stable steady state $c$ over arbitrarily long but finite time
intervals. The deterministic dynamics have $c$ as a stable steady state, so
any deviation from $c$ necessarily carries a positive cost. The cost increases
the ``further away'' one deviates, so the cost is a natural quantity which
increases as one moves away from $c$, which is precisely a Lyapunov function
type property.

Though this property is implicit in the work of Freidlin and Wentzell, we
provide a proof for completeness.

\begin{proposition}\label{prop:quasipotential_lyapunov}
  Let $c$ be a stable steady state for the system $\dot{x} (t) = D_p H (x (t),
  0)$. The quasipotential $Q (x) \doteq Q (c, x)$ is a Lyapunov function for
  this system in the sense of Definition \ref{def:Lyapunov_function} on the
  domain of attraction $D$ of $c$.
\end{proposition}

\begin{proof}
  Let $x \in D$. We first show that $Q (x) \geqslant 0$ with equality if and
  only if $x = c$. Since $c$ is assumed to be a steady state, we have $D_p H
  (c, 0) = 0$. By properties of the Legendre transform, $L (c, 0) = 0$, so any
  trajectory which stays at $c$ for a positive amount of time achieves the
  minimum in the variational definition of the quasipotential and incurs zero
  cost. Thus $Q (c) = 0$.
  
  Next, for each $x \in \mathbb{R}^d$ the unique $b (x) \in \mathbb{R}^d$
  such that $L (x, b (x)) = 0$ is $b (x) = D_p H (x, 0)$. Since $c$ is a
  stable steady state, it can be shown that any absolutely continuous
  trajectory $\varphi$ which leads from $c$ to $x$ in time $T$ must
  necessarily satisfy $\dot{\varphi} (t) \neq D_p H (\varphi (t), 0)$ for all
  $t$ in some interval. In particular, $\varphi$ must incur a positive cost
  over some time interval, hence $Q (x) > 0$ for all $x \neq c$.
  
  It remains to show that $Q (x (t + h)) - Q (x (t)) \leqslant 0$ for all $t
  \geqslant 0$ and $h > 0$, with equality if and only if $x (t) = c$. Let
  $\varepsilon > 0$ be arbitrary and $T < \infty$, and let $\varphi \in D ([0,
  T])$ be such that $\varphi (0) = c$, $\varphi (T) = x (t)$, and $I_T
  (\varphi) < Q (x (t)) + \varepsilon$. Define a new trajectory $\psi \in D
  ([0, T + h])$ by concatenating it with $\varphi$ on $[0, T]$ and with the
  trajectory $\{ x (s) : s \in [t, t + h] \}$ on $(T, T + h]$. Since $x
  (\cdot)$ follows the noiseless dynamics, it has zero cost and $I_T (\psi) =
  I_T (\varphi)$. Since $\psi (0) = c$ and $\psi (T + h) = x (t + h)$, it
  follows that $Q (x (t + h)) \leqslant Q (x (t)) + \varepsilon$. Since
  $\varepsilon > 0$ was arbitrary, $Q (x (t + h)) \leqslant Q (x (t))$. See Figure \ref{fig:proof} for a mental picture. 
  \begin{figure}
  \centering
  \includegraphics[scale=0.35]{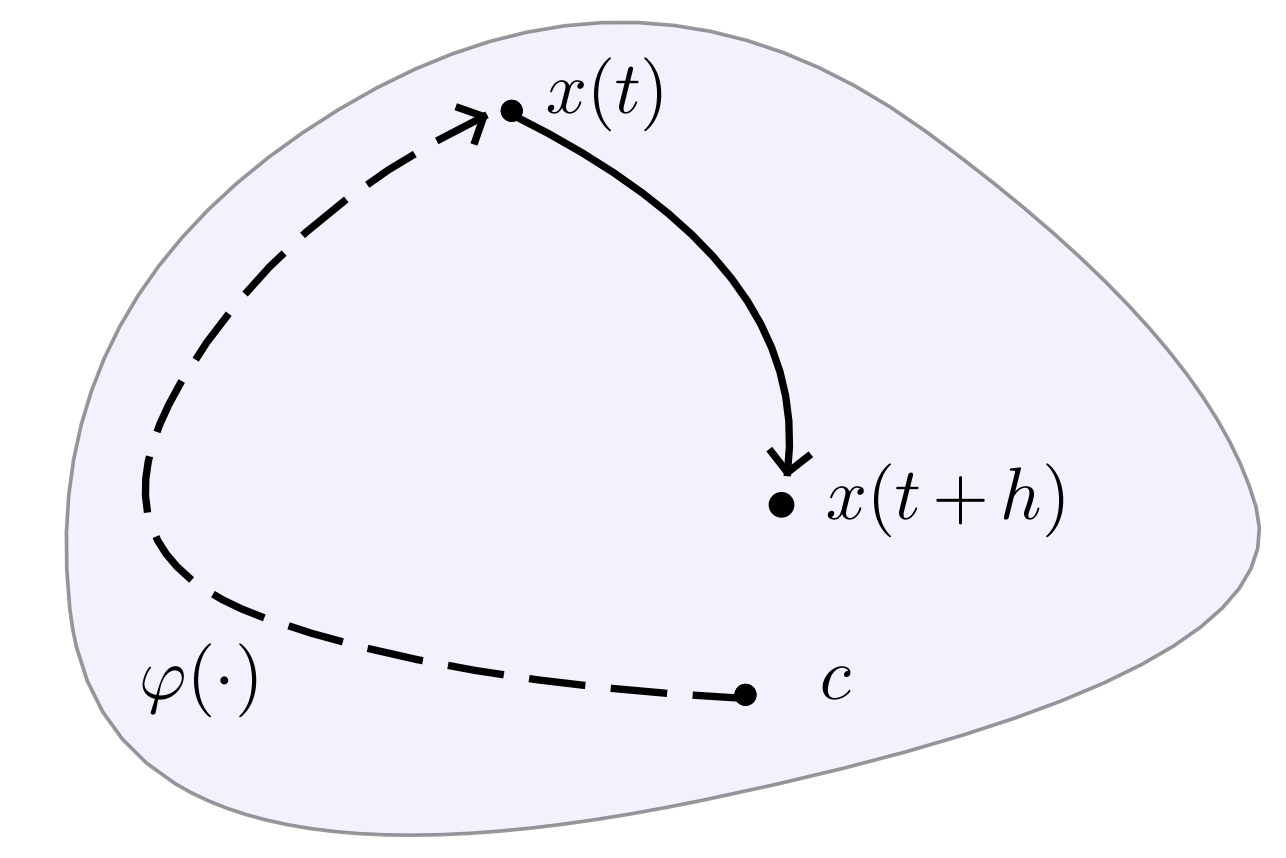}
  \caption{The trajectory $\psi(\cdot)$ defined on $[0, T+h]$ is the concatenation of $\varphi$ on $[0, T]$ with $x(\cdot)$ on $(T, T+h]$. }\label{fig:proof}
  \end{figure}
  
  To rule out the case of equality, note that if there are $t \geqslant 0$ and
  $h > 0$ such that $Q (x (t + h)) = Q (x (t))$, then by the mean value
  theorem there is $\xi \in (t, t + h)$ such that $\langle D Q (x (\xi)),
  \dot{x} (\xi) \rangle h = 0$. Since $c$ is the unique stable steady state in
  $D$, $\dot{x} (\xi) \neq 0$ and $D Q (x (\xi)) \neq 0$ unless $x (\xi) = c$.
  Thus equality cannot hold unless $x (\xi) = c$. This completes the proof. 
\end{proof}

To illustrate the role of the quasipotential as a Lyapunov function, consider
the following example.

\begin{example}
  \label{ex:quasipotential_lyapunov}Consider the reaction network $A
  \rightleftarrows \emptyset$ from Example \ref{ex:guiding_example}. We will
  see in Examples \ref{ex:quasipotential_invariant} and
  \ref{ex:quasipotential_hjb} that the quasipotential for this CRN is $Q
  (x) = x \log x - x + 1$. This is a Lyapunov function for the deterministic
  dynamics $\dot{x} (t) = 1 - x (t)$ relative to the stable steady state at $c
  = 1$. Since $c$ is globally attracting, $Q$ is a Lyapunov function on all of
  $\mathbb{R}_+$. Indeed, $Q (1) = 0$, $Q (x) > 0$ for $x \neq 1$, and
  \begin{eqnarray*}
    \frac{d}{d t} Q (x (t)) & = & \dot{x} (t) \log x (t)\\
    & = & (1 - x (t)) \log x (t) .
  \end{eqnarray*}
  If $x (t) \in [0, 1)$ then $\log x (t) < 0$, whereas if $x (t) \in (1,
  \infty)$ we have $(1 - x (t)) < 0$. In either case, the above expression is
  negative, with equality if and only if $x (t) = 1$.
\end{example}

\subsection{Quasipotential and invariant
measure}\label{sec:quasipotential_invariant}

The following is a classical result of Freidlin-Wentzell theory; see for
instance Theorem 4.3 in Chapter 4 of {\cite{FreidlinWentzell}} for the case of
diffusions. The modification of the proof for jump processes is relatively straightforward, as the estimates depend on ``uniform'' properties of the rate function; see p. 74 of \cite{FreidlinWentzell}. This relationship between the invariant distribution and the quasipotential is what was used to determine the Lyapunov
functions in {\cite{AndersonLyapunov15}}. 

\begin{theorem}
  Let $c$ be the unique stable steady state of the system $\dot{x} (t) = D_p H
  (x (t), 0)$, and suppose that its domain of attraction is all of
  $\mathbb{R}^d$. Assume that for each $n \geqslant 1$, $X^n$ has a unique
  invariant measure $\pi^n$. Then
  \[ \lim_{n \rightarrow \infty} \frac{1}{n} \log \pi^n (A) = - \inf_{x \in A}
     Q (c, x), \]
  where $A \subset \mathbb{R}^d$ is any domain with compact boundary $\partial
  A$ common for both $A$ and the closure of $A$.
\end{theorem}

\begin{example}
  \label{ex:quasipotential_invariant}The sequence of processes $\{ X^n, n
  \geqslant 1 \}$ associated to the reaction network $A \rightleftarrows
  \emptyset$ are birth-death process with state-dependent jump rates as
  defined by the generator in Example \ref{ex:generator}, and the density of
  its unique stationary distribution is given by
  \[ \pi^n (x) = e^{- n}  \frac{n^{n x}}{(n x) !}, \quad x \in n^{- 1}
     \mathbb{Z}^d_+ . \]
  Applying the large deviations scaling and using Stirling's formula, $\log k!
  = k \log k - k + O (\log k)$ for $k \geqslant 1$, we find
  \[ \lim_{n \rightarrow \infty} - \frac{1}{n} \log \pi^n (x) = x \log x - x +
     1. \]
\end{example}

We note that Example \ref{ex:quasipotential_invariant} illustrates a technique
which was used to determine \ Lyapunov functions for the deterministic
dynamics in {\cite{AndersonLyapunov15}}.

\subsection{Quasipotential and HJB}\label{sec:quasipotential_hjb}

The quasipotential is related to the HJB in the following way.

\begin{theorem}
  \label{thm:hjb_to_quasipotential} Let $c$ be a stable steady state for the
  noiseless dynamics $\dot{x} (t) = D_p H (x (t), 0)$, and let $D$ denote
  the domain of attraction of $c$. Let $U : D \rightarrow [0, \infty)$ be
  continuously differentiable on $D$ and continuous on its closure $\bar{D}$,
  with $U (x) > 0$ and $D U (x) \neq 0$ for $x \neq c$, and $U (c) = 0,
  D U (c) = 0$. Moreover, suppose that for all $x \in D$,
  \[ H (x, D U (x)) = 0. \]
  Then $U (x) = Q (c, x)$ for all $x \in \{ z : U (z) \leqslant \min_{y \in
  \partial D} U (y) \}$.
\end{theorem}

The proof of this proposition can be found in Theorem 4.3, Chapter 5 of
{\cite{FreidlinWentzell}}. Since the HJB PDE allows us to identify the
quasipotential and the quasipotential is always a Lyapunov function, the HJB
can be used to identify Lyapunov functions for the noiseless dynamics. The
example below shows how the HJB can be used to identify the Lyapunov function
$x \log x - x + 1$ introduced in Example \ref{ex:quasipotential_lyapunov}.

\begin{example}
  \label{ex:quasipotential_hjb}Consider the time-independent HJB for the CRN
  $A \rightleftarrows \emptyset$,
  \[ H (x, D U (x)) = x (e^{- D U (x)} - 1) + e^{D U (x)} - 1 = 0. \]
  One can check by direct calculation that $U (x) = x \log x - x + 1$
  satisfies $H (x, D U (x)) = 0$ for all $x \in \mathbb{R}_+$. Moreover, $U
  (x) \geqslant 0$ with equality if and only if $x = 1$, and the same is true
  for $D U (x) = \log x$. $U$ is also smooth and the domain of attraction is
  all of $\mathbb{R}_+$, so $U$ satisfies the conditions of Theorem
  \ref{thm:hjb_to_quasipotential} for all $x \in \mathbb{R}_+$ and is equal to
  the quasipotential everywhere. 
\end{example}

\section{Quasipotential for complex-balanced
systems}\label{sec:complex_balanced}

In this section we consider the quasipotential for the class of
``complex-balanced'' reaction networks. We will show that the time-independent
Hamilton-Jacobi-Bellman equation can be used to characterize such networks, as
it will admit a solution of a specific form if and only if the network is
complex-balanced.

\subsection{Complex-balanced networks}

We first define a structural property of reaction networks.

\begin{definition}
  \label{def:weakly_reversible}A CRN $G = (\mathcal{A}, \mathcal{C},
  \mathcal{R})$ is said to be {\emph{weakly reversible}} if for any $(a, b)
  \in \mathcal{R}$, there is a sequence of reactions $(b_0, b_1), (b_1, b_2),
  \ldots, (b_{k - 1}, b_k) \in \mathcal{R}$ such that $b_0 = b$ and $b_k = a$.
  
\end{definition}

We now define complex-balanced networks. The notion was introduced in
{\cite{HornJackson72}} and has been extensively studied since
{\cite{Horn72, Feinberg72, Feinberg95, AndersonProductForm10, AndersonLyapunov15}.}

\

\begin{definition}
  A CRN $G = (\mathcal{A}, \mathcal{C}, \mathcal{R})$ is said to be
  complex-balanced for a choice of rate constants $\{ \kappa_{a, b} : (a, b)
  \in \mathcal{R} \}$ if, for any $z \in \mathcal{C}$,
  \begin{equation}
    \sum_{a : (a, z) \in \mathcal{R}} \kappa_{a, z} c^a = \sum_{b : (z, b) \in
    \mathcal{R}} \kappa_{z, b} c^z . \label{eqn:complex_balanced}
  \end{equation}
\end{definition}

In other words, complex-balanced systems are ones which admit a steady state
where the net flow in each \emph{complex} is zero. This is a stronger
requirement than simply being a steady state, which requires only that the net
flow in each \emph{species} is zero.

There are two important consequences of a network being complex-balancing. The
first is that a complex-balanced network must be weakly reversible in the
sense of Definition \ref{def:weakly_reversible}. The second is that
complex-balanced networks admit a unique steady state in each stoichiometric
compatibility class, and that this steady state is stable. The stability of
the steady state is typically established using the well-known Lyapunov
function
\begin{equation}
  V (c, x) = \sum_{i = 1}^d c_i \ell (x_i / c_i)
  \label{eqn:complex_balanced_hjb_solution},
\end{equation}
where $\ell (x) = x \log x - x + 1$. We will show in the next section that a
stronger characterization of $V$ is that it is a solution to the HJB PDE
corresponding to a complex-balanced reaction network, which by Proposition
\ref{prop:quasipotential_lyapunov} automatically guarantees it is a Lyapunov
function.

\subsection{Main theorems}
We first show that \eqref{eqn:complex_balanced_hjb_solution} is a solution to the HJB if and only if the network is complex-balanced. We note that this has already been observed in \cite{Zhou19}, Theorem 11. There, the approach taken is the typical approximation of the Fokker-Planck equation with a LLN scaling\footnote{This is typically referred to as a ``WKB approximation'': one passes from the Schr\"odinger equation to the HJB equation in the ``small noise'', i.e. classical limit.} 
\begin{theorem}
  \label{thm:complex_balanced_hjb}Let $G = (\mathcal{A}, \mathcal{C},
  \mathcal{R})$ be a reaction network with rate constants $\kappa = \{
  \kappa_{a, b} > 0 : (a, b) \in \mathcal{R} \}$, and let $c$ be a steady
  state in any stoichiometric compatibility class. Then $V (c, x)$ as defined
  in {\eqref{eqn:complex_balanced_hjb_solution}} satisfies $H (x, D_x
  V (c, x)) = 0$ for all $x \in \mathcal{S}$ if and only if the state $c$ is
  complex-balancing. 
\end{theorem}

\begin{proof}
  The proof is a straightforward calculation. Suppose first that $c \in
  \mathbb{R}^d_{> 0}$ satisfies, for each $z \in \mathcal{C}$, the
  complex-balanced relation
  \[ \sum_{a : (a, z)\in\mathcal{R}} \kappa_{a, z} c^a = \sum_{b : (z, b)\in\mathcal{R}} \kappa_{z, b} c^z
     . \]
  We verify that $V$ satisfies the steady-state HJB by direct calculation.
  \begin{eqnarray*}
    H (x, D V (x)) & = & \sum_{(a, b) \in \mathcal{R}} \kappa_{a, b} x^a 
    (e^{\langle \zeta_{a, b}, D V (x) \rangle} - 1)\\
    & = & \sum_{(a, b) \in \mathcal{R}} \kappa_{a, b} x^a \left( \frac{\left(
    \frac{x}{c} \right)^b}{\left( \frac{x}{c} \right)^a} - 1 \right)\\
    & = & \sum_{(a, b) \in \mathcal{R}} \kappa_{a, b} x^b  \frac{c^a}{c^b} -
    \sum_{(a, b) \in \mathcal{R}} \kappa_{a, b} x^a\\
    & = & \sum_{(a, b) \in \mathcal{R}} \left( \left[ \sum_{z \in
    \mathcal{C}} 1_{\{ b = z \}} \right] \kappa_{a, b} x^b  \frac{c^a}{c^b}
    \right) - \sum_{(a, b) \in \mathcal{R}} \left( \left[ \sum_{z \in
    \mathcal{C}} 1_{\{ a = z \}} \right] \right) \kappa_{a, b} x^a\\
    & = & \sum_{z \in \mathcal{C}} \left[ \sum_{a : (a, z) \in \mathcal{R}}
    \kappa_{a, z} x^z  \frac{c^a}{c^z} - \sum_{b : (z, b) \in \mathcal{R}}
    \kappa_{z, b} x^z \right]\\
    & = & \sum_{z \in \mathcal{C}} \left[ \frac{x^z}{c^z} \sum_{a : (a, z)
    \in \mathcal{R}} \kappa_{a, z} c^a - \sum_{b : (z, b) \in \mathcal{R}}
    \kappa_{z, b} x^z \right]\\
    & = & \sum_{z \in \mathcal{C}} \left[ \frac{x^z}{c^z} \sum_{b : (z, b)
    \in \mathcal{R}} \kappa_{z, b} c^z - \sum_{b : (z, b) \in \mathcal{R}}
    \kappa_{z, b} x^z \right]\\
    & = & 0.
  \end{eqnarray*}
  In the penultimate equality we used the complex-balanced relation.
  
  For the reverse direction, observe that until the penultimate equality we
  did not use the complex-balanced relation. The relation $H (x, D V (x))
  = 0$ for all $x \in \mathcal{S}$ implies
  \[ \sum_{z \in \mathcal{C}} \left[ x^z \left( \sum_{a : (a, z) \in
     \mathcal{R}} \kappa_{a, z}  \frac{c^a}{c^z} - \sum_{b : (z, b) \in
     \mathcal{R}} \kappa_{z, b}  \right)  \right] = 0 \]
  for all $x \in \mathcal{S}$. Since each $z \in \mathcal{C}$ corresponds to a
  different monomial $x^z$, for the above expression to be identically zero on
  all of $\mathcal{S}$, we must have each term in the sum over $z$ being zero.
  Thus
  \[ \sum_{a : (a, z) \in \mathcal{R}} \kappa_{a, z}  \frac{c^a}{c^z} -
     \sum_{b : (z, b) \in \mathcal{R}} \kappa_{z, b} = 0, \]
  which, after multiplying by $c^z$, is the complex balanced relation
  {\eqref{eqn:complex_balanced}}.
\end{proof}

\begin{corollary}
  The network is complex-balanced if and only if $Q (c, x) = V (c, x)$ on a
  neighbourhood of $c$. 
\end{corollary}

\begin{proof}[Sketch of proof.]
  Suppose that $Q (c, x) = V (c, x)$ for all $x$ in some neighbourhood $N$ of
  $c$. By restricting to a smaller neighbourhood if necessary, we may assume that $N$ is contained in the
  domain of attraction of $c$. Since $Q = V$ is the quasipotential, it solves
  the steady-state HJB $H (x, D V (x)) = 0$ for $x \in N$ (see e.g. \cite{Nyquist14}). By Theorem
  \ref{thm:complex_balanced_hjb}, $c$ satisfies the complex-balanced
  relations, and so the network must be complex-balanced.
  
  Conversely, suppose that the network is complex-balanced. Then by Theorem
  \ref{thm:complex_balanced_hjb}, $V$ solves the HJB PDE for all $x \in
  \mathcal{S}_c$. Moreover, $V (c, x)$ is non-negative, and $V (c, x)$ and
  $D V (c, x)$ equal respectively the scalar and vector 0 if and only if
  $x = c$. Thus $V$ satisfies the conditions of Theorem
  \ref{thm:hjb_to_quasipotential}, and so $V$ equals the quasipotential on $N
  = \{ x : V (c, x) < \min_{y \notin \text{int} (\mathcal{S}_c)} V (c, y) \}$.
  (Here, $\text{int} (\mathcal{S}_c)$ refers to the interior of
  $\mathcal{S}_c$.)
\end{proof}

\section{Examples}\label{sec:examples}

\subsection{Calculations in \cite{AndersonLyapunov15}}\label{sec:examples_anderson}

We demonstrate how Theorem 9 in {\cite{AndersonLyapunov15}} can be easily
derived from the theory developed in Section \ref{sec:quasipotential_hjb}. We
formulate the statement in a way that is more natural for the quasipotential.

\begin{proposition}
  Let $G = (\mathcal{S}, \mathcal{C}, \mathcal{R})$ be a one-dimensional
  birth-death CRN with $\mathcal{S}= \{ A \}$ and rate constants $\{
  \kappa_{a, b} \}_{(a, b) \in \mathcal{R}}$. Let $D_c$ denote the domain of
  attraction of $c > 0$. Then, for $x \in D_c$, the gradient of the
  quasipotential relative to $c$ is
  \begin{equation}
    D Q (c, x) = \int_c^x \log \left( \frac{\sum_{(n, n - 1) \in
    \mathcal{R}} \kappa_{n, n - 1} s^n_{}}{\sum_{(n, n + 1) \in \mathcal{R}}
    \kappa_{n, n + 1} s^n} \right) \text{ds} .
    \label{eqn:quasipotential_birth_death}
  \end{equation}
\end{proposition}

\begin{proof}
  Fix $c > 0$. The equation $H (x, p (x)) = 0$ can be made algebraic by using
  the substitution $y (x) \doteq e^{p (x)}$ and clearing the denominator of
  $y$'s. For $x \in D_c$, consider
  \begin{eqnarray*}
    K (x, y) & \doteq & y H (x, \log y)\\
    & = & \sum_{(n, n + 1) \in \mathcal{R}} \kappa_{n, n + 1} x^n (y^2 - y) +
    \sum_{(n, n - 1) \in R} \kappa_{n, n - 1} x^n (1 - y) .
  \end{eqnarray*}
  Write $A = A (x) = \sum_{(n, n + 1) \in \mathcal{R}} \kappa_{n, n + 1} x^n$
  and $B = B (x) = \sum_{(n, n - 1)} \kappa_{n, n - 1} x^n$. Then
  \[ K (x, y) = A (x) y^2 - (A (x) + B (x)) y + B (x) . \]
  Viewing this as a quadratic equation in $y$, we find the pair of solutions
  to be
  \[ y_{\pm} = \frac{1}{2 A} \left( A + B \pm \sqrt{(A - B)^2} \right) . \]
  There are four cases and two possible solutions, depending on whether we
  take the $+$ or $-$, and whether $A > B$ or $A \leqslant B$. We have
  
  \begin{center}
    \begin{tabular}{lll}
      & $A > B$ & $A \leqslant B$\\
      $+$ & $1$ & $B / A$\\
      $-$ & $B / A$ & $1$
    \end{tabular}.
  \end{center}
  
  Since $p (x) = \log (y (x))$, we find that a solution to the original
  equation is
  \[ p (x) = \log \left( \frac{B (x)}{A (x)} \right), \]
  which is the same as {\eqref{eqn:quasipotential_birth_death}}. It remains
  only to verify that this solution satisfies the conditions of Theorem
  \ref{thm:hjb_to_quasipotential}. Observe that $\Phi (x) = 0$ if and only if
  $A (x) = B (x)$. If $c$ is stable, then $\Phi (c) > 0$ for $x < c$, $x \in
  D_c$, and $\Phi (c) < 0$ for $x > c, x \in D_c$, which implies respectively
  that $A (x) < B (x)$ and $A (x) > B (x)$. Consequently, $p (x) < 0$ for $x <
  c, x \in D_c$ and $p (x) > 0$ for $x > c, x \in D_c$. This is all that needs
  to be checked.
  
  \ 
\end{proof}

As another illustration, consider Example 13 in {\cite{AndersonLyapunov15}}.

\begin{example}
  Consider the reaction network
  \[ A \rightarrow \emptyset \rightarrow 2 A, \]
  with associated Hamiltonian
  \[ H (x, p) = x (e^{- p} - 1) + e^{2 p} - 1. \]
  This is a cubic equation in $y (x) = e^{p (x)}$, a solution of which is
  \[ y (x) = \frac{1}{2} \left( \sqrt{1 + 4 x} - 1 \right) . \]
  The other solutions are $y (x) = - \left( \sqrt{1 + 4 x} + 1 \right) / 2$
  and $y (x) = 1$, the first of which is not admissible because $y (x) = e^{p
  (x)} \geqslant 0$, while the second corresponds to the deterministic system
  $\dot{x} = \Phi (x)$.
  
  One can conclude that the gradient of the quasipotential $Q$ satisfies
  \[ D Q (x) = \log \left( \sqrt{1 + 4 x} - 1 \right) - \log 2. \]
  Indeed,
  \begin{eqnarray*}
    &  & H (x, D Q)\\
    & = & x \left( \frac{1}{\frac{1}{2} \left( \sqrt{1 + 4 x} - 1 \right)} -
    1 \right) + \left( \frac{1}{2} \right)^2 \left( \sqrt{1 + 4 x} - 1
    \right)^2 - 1\\
    & = & 0.
  \end{eqnarray*}
\end{example}

Our procedure is much simpler than the one outlined in
{\cite{AndersonLyapunov15}}. The same style of calculation can be repeated for
each example in {\cite{AndersonLyapunov15}}, and in each case we find that the
steady-state Hamilton-Jacobi-Bellman equation is satisfied.

\subsection{Examples from symplectic and K{\"a}hler
geometry}\label{sec:examples_kahler}

While the results developed so far clarify why an appropriately scaled limit
of the invariant distribution identifies the quasipotential, they are simply a
recasting of known results in probabilistic terms. We suspect, however, that
there is some underlying geometric structure which describes the level curves
which determine the quasipotential, i.e. $(x, p) : H (x, p) = 0$.

One of the clues we have found most compelling is a series of papers by
Miguel Abreu, which builds on work of Delzant and Guillemin. The
main idea is to view the stoichiometric compatibility class as a convex
polytope, which can then be seen as the image of a moment map of an effective
Hamiltonian action . Theorem 2.8 in {\cite{Abreu08}} then says that there is a
``canonical'' toric complex structure determined by the combinatorial data
which defines the polytope. The associated Riemannian K{\"a}hler metric is then
determined by a potential function, which turns out to be closely related to
the quasipotential.

It is important to note that the quasipotential carries information only about certain aspects of the CRN, since it solves the HJB for only a particular class of infinite-horizon problems. We mention this because of related work on toric varieties in CRNs by \cite{Sturmfels09, ShiuThesis10}. The authors demonstrate that the algebraic relationships between the variables $x$ and rate constants $\kappa$ imposed by the equation $D_p H(x, 0) = 0$ (see \eqref{eqn:deterministic_dynamics}) define an ideal that corresponds to a toric variety if the network is complex-balanced. Though the observations here might be related, we note that two different CRNs can have the same quasipotential, but whose algebraic varieties $D_p H(x, 0) = 0$ can look very different.

We present three examples. The first two illustrate the issue of viewing the
stoichiometric compatibility classes as convex polytopes. We formally change
coordinates to show that the quasipotential for complex balanced networks
agrees with the K{\"a}hler potential on an appropriately reparametrized
polytope. The last example is somewhat unrelated, but provides a ``classical
mechanics'' perspective on the network $A \rightleftarrows \emptyset$. We are
compelled to include this example because there is an elegant transformation
which maps the corresponding Hamiltonian into that of an ``antiharmonic
oscillator''. 

\begin{definition}
  The potential introduced in Example \ref{ex:kahler1} is defined in terms of
  combinatorial data by
  \begin{equation}
    g_P (x) = \sum_{i = 1}^k (\langle \mu_i, x \rangle - \lambda_i) \log
    (\langle \mu_i, x \rangle - \lambda_i), \label{eqn:kahler_potential}
  \end{equation}
  where $\mu_i \in \mathbb{Z}^d$ and $\lambda_i \in \mathbb{R}$. The polytope is the set of points $x$ where $\langle \mu_i, x \rangle \geq
  \lambda_i$,
  \[ P\doteq \{a_i(x) \geq 0, i = 1, \dots, k\}, \quad a_i(x) \doteq  \langle \mu_i, x \rangle - \lambda_i,\]
  and for points in the interior of $P$, $g_P (x)$ is well-defined and smooth. There
  are additional combinatorial assumptions on this setup and the reader is
  referred to {\cite{Abreu08}} for details.
\end{definition}

\begin{remark}
  The potential $g_P$ is strictly convex on the interior of its domain of
  definition, and in particular, $G_P (x) \doteq \text{Hess} (g_P (x))$ is
  positive definite there. In particular, $G_P$ defines a Riemannian metric.
  It is this metric which we suspect is related to the notion of ``cost'' or
  ``distance'' induced by the quasipotential. Note that the Hessian with
  respect to $x$ of $V (c, x)$ as defined in
  {\eqref{eqn:complex_balanced_hjb_solution}} is independent of $c$.
\end{remark}

\begin{example}
  \label{ex:kahler1}Consider the reaction network
  \[ A_1 \rightleftarrows A_2, \]
  with both rate constants set equal to 1. It is easy to verify that this
  network is complex balanced, and that $c = (1, 1)$ is the unique stable
  steady state in its stoichiometric compatibility class $\mathcal{S}_c = \{ x
  \in \mathbb{R}^2_+ : x_1 + x_2 = 2 \}$. The quasipotential relative to $c$
  is
  \[ Q (x) = \ell (x_1) + \ell (x_2) = x_1 \log x_1 + x_2 \log x_2, \]
  where $\ell (z) = z \log z - z + 1$ for $z \geqslant 0$. Since $x_1 + x_2 =
  2$ and $x_i \in [0, 2]$ for $i = 1, 2$, we can parametrize $x = (x_1, x_2)$
  in terms of a single variable, say $y \in [- 1, 1]$. Let $x_1 = 1 - y$ and
  $x_2 = 1 + y$. Then
  \[ Q (x (y)) = (1 - y) \log (1 - y) + (1 + y) \log (1 + y) . \]
  We will compare this network with Example 2.6 in {\cite{Abreu08}}, which we
  transcribe here for convenience. Let $M = S^2 \subset \mathbb{R}^3$ be the
  standard unit two-sphere, equipped with its standard symplectic form, which
  is the standard area form with total area $4 \pi$. Consider the $S^1$ action
  on $M$ given by rotation around any axis. The moment map $\phi : M
  \rightarrow \mathbb{R}$ is the projection map to this axis of rotation, so
  the moment polytope is $P = [- 1, 1]$. This polytope is determined by two
  affine functions,
  \[ a_{- 1} (x) = 1 + x, \quad a_1 (x) = 1 - x, \]
  and the potential function $g_P$ defined on the interior of $P$ is
  \[ g_P (y) = \frac{1}{2} [a_1 (y) \log (a_1 (y)) + a_2 (y) \log (a_2 (y))] =
     \frac{1}{2} Q (y) . \]
\end{example}

\begin{example}
  Consider the network
  \[ A_1 + A_2 \rightleftarrows 2 A_2 \rightleftarrows A_2 + A_3 . \]
  It is easy to verify that this network is also complex-balanced, and that $c
  = (1, 1, 1)$ is the unique stable steady state in it stoichiometric
  compatibility class. The quasipotential on $\mathcal{S}_c = \left\{ x :
  \sum_{i = 1}^3 x_i = 3 \right\}$ is therefore
  \[ Q (x) = x_1 \log x_1 + x_2 \log x_2 + x_3 \log x_3 . \]
  Using the relation $x_1 + x_2 + x_3 = 3$, we can re-express the
  quasipotential in terms of two variables $y_1, y_2$, where $x_3 = 3 - y_1 -
  y_2$. Then
  \begin{eqnarray*}
    Q (x (y)) & = & y_1 \log y_1 + y_2 \log y_2 + (3 - y_1 - y_2) \log (3 -
    y_1 - y_2),
  \end{eqnarray*}
  which is exactly what one would obtain from the K{\"a}hler potential
  {\eqref{eqn:kahler_potential}} using $\mu_1 = (1, 0)$, $\lambda_1 = 0$,
  $\mu_2 = (0, 1)$, $\lambda_2 = 0$, and $\mu_3 = (- 1, - 1)$ and $\lambda_3 =
  - 3$. These three pairs define the convex polytope $\{ y \in \mathbb{R}^2 :
  y_1 \geqslant 0, y_2 \geqslant 0, y_1 + y_2 \leqslant 3 \}$, which is the
  projection of the simplex $\mathcal{S}_c$ onto the $(x_1, x_2)$ plane. 
\end{example}

\

\begin{example}
  In this example we use basic techniques from classical mechanics. A standard
  reference is {\cite{Arnold}}. Consider again the reaction network from
  Example \ref{ex:guiding_example},
  \[ A \rightleftarrows \emptyset, \]
  with both rate constants set equal to 1. The Hamiltonian is
  \[ H (x, p) = x (e^{- p} - 1) + e^p - 1. \]
  We initially sought to identify the minimizers in $A_T (x, y)$, the space of
  trajectories $\varphi$ which go from $x$ to $y$ in time $T$ (see
  {\eqref{eqn:A_T}}), and we conjectured that if $(x (\cdot), p (\cdot))$
  satisfies Hamilton's equations of motion
  {\eqref{eqn:equation_of_motion_1}}-{\eqref{eqn:equation_of_motion_2}}
  then $x (\cdot)$ would be a minimizing trajectory in $A_T (x, y)$. Since the
  Hamiltonian $H$ is ``time-independent'', we have the \emph{conservation
  of energy} relation along curves $(x (\cdot), p (\cdot))$,
  \[ \frac{d}{d t} H (x (t), p (t)) = D_x H (x (t), p (t))  \dot{x} (t) + D_p
     H (x (t), p (t))  \dot{p} (t) = 0, \label{eqn:conservation_of_energy} \]
  which can be seen by using the equations for $\dot{x} (t)$ and $\dot{p}
  (t)$. Thus, we expect the pairs $(x (\cdot), p (\cdot))$ to live on level
  sets $L_E \doteq \{ (x, p) : H (x, p) = E \}$. If we can determine the
  relationship between the level sets and points $(x (t), p (t))$, $t \in [0,
  T]$, then we can have a better understanding of the energy landscape of the
  network.
  
  \paragraph{Harmonic oscillator.}
  
  To illustrate the approach we will take with the CRN, consider perhaps the
  most well-studied Hamiltonian system is the \emph{harmonic oscillator},
  \[ H_{\text{osc}} (x, p) = \frac{x^2}{2} + \frac{p^2}{2}, \]
  with equations of motion
  \begin{eqnarray}
    \dot{x} (t) & = & p  \label{eqn:ho1}\\
    \dot{p} (t) & = & - x.  \label{eqn:ho2}
  \end{eqnarray}
  The level sets of $H_{\text{osc}}$ are ellipses (in this case, circles).
  While the pair {\eqref{eqn:ho1}}-{\eqref{eqn:ho2}} can be solved directly,
  part of the success of classical mechanics relies on parametrizing systems
  in terms of conserved quantities. Consider the change of variables
  \begin{eqnarray*}
    x (E, Q) & = & \sqrt{2 E} \cos (Q) \\
    p (E, Q) & = & \sqrt{2 E} \sin (Q) .
  \end{eqnarray*}
  Let $J = \frac{\partial (x, p)}{\partial (E, Q)}$ denote the Jacobian of
  this transformation. It turns out that the above change of variables is
  \emph{symplectic}, in the sense that
  \[ J^T \Omega J = \Omega, \]
  where
  \[ \Omega = \left(\begin{array}{cc}
       0 & 1\\
       - 1 & 0
     \end{array}\right) . \]
  Symplectic transformations preserve Hamilton's equations, and in this case
  the new Hamiltonian is $H_{\text{new}} (E, Q) = E$.
  
  \paragraph{Antiharmonic oscillator.}
  
  Consider now the \emph{antiharmonic oscillator},
  \[ H_{\text{anti}} (x, p) = \frac{p^2}{2} - \frac{x^2}{2} . \]
  Similarly to the previous example, we can change variables to obtain a new
  Hamiltonian which depends only the ``energy level'' $E$. In this case we
  would use hyperbolic trigonometric functions and see that the level sets are
  hyperbolas.
  
  \paragraph{Back to $A \rightleftarrows \emptyset$.}
  
  The level sets of the Hamiltonian are displayed in Figure
  \ref{fig:level_sets}.
  
  \begin{figure}[h]
  \centering
\includegraphics[scale=0.7]{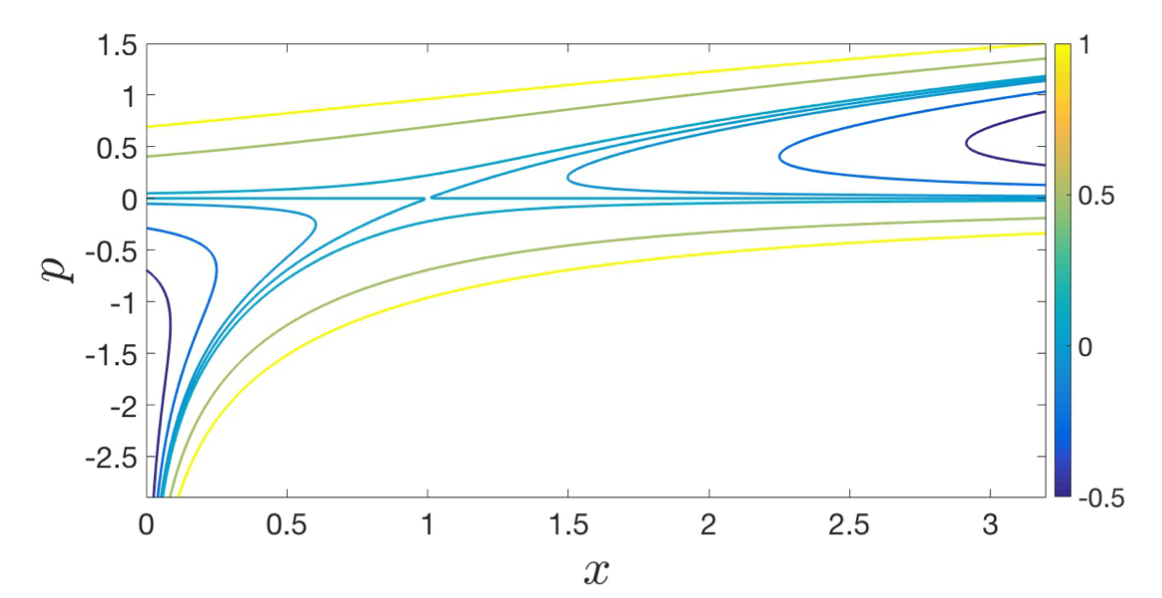}
    \caption{Level sets of the Hamiltonian $H (x, p) = x (e^{- p} - 1) + e^p -
    1$.}\label{fig:level_sets}
  \end{figure}
  
  The level curves corresponding to $H = 0$ look like the asymptotes of a
  hyperbola. We will make this evident. Consider the change of variables
  \begin{eqnarray*}
    x (X, P) & = & (X + 1) (P + 1)\\
    p (X, P) & = & \log (P + 1) .
  \end{eqnarray*}
  It turns out this change of variables is also symplectic, and the new
  Hamiltonian is given by
  \[ H_{\text{new}} (X, P) = - X P. \]
  If $(x, p) \in \mathbb{R}_+ \times \mathbb{R}$, then $(X, P) \in [- 1,
  \infty) \times (- 1, \infty)$. It is straightforward to find a symplectic
  transformation which turns this Hamiltonian into the antiharmonic
  oscillator, though not necessarily with real coordinates. For more details,
  see for instance {\cite{Arnold}}.
\end{example}

\bibliographystyle{unsrt}
\bibliography{main} 

\end{document}